\documentclass[12pt,oneside]{amsart}

\usepackage[a4paper]{geometry}  

\usepackage{color}

\usepackage{amssymb}

\usepackage[matrix,arrow]{xy}

\usepackage{graphicx}
\usepackage{ifpdf}

\usepackage[english]{babel}

\usepackage{hyperref}

\usepackage{mathptmx}

\theoremstyle{plain}
 \newtheorem{theorem}{Theorem}[section]
\theoremstyle{definition}
 \newtheorem{definition}[theorem]{Definition}

\theoremstyle{plain}
 \newtheorem{proposition}[theorem]{Proposition}
 \newtheorem{corollary}[theorem]{Corollary}
 \newtheorem{lemma}[theorem]{Lemma} 
\theoremstyle{remark}

\def\Z{\mathbb{Z}}
\def\N{\mathbb{N}}

\newcommand{\K}{\mathcal K}
\newcommand{\rank}{{\rm rank}}

\definecolor{m}{rgb}{1,0.1,1}

\title{Isomorphism versus commensurability for a class of finitely presented groups}

\date{\today}

\author{Goulnara Arzhantseva}
\address{University of Vienna,
Faculty of Mathematics, Nordbergstra{\rm \ss}e 15, 1090 Vienna,
Austria} \email{goulnara.arzhantseva@univie.ac.at}

\author{Jean-Fran\c{c}ois Lafont}
\address{The Ohio State University, Department of Mathematics,
 100 Math Tower, 231 West 18th Avenue,
Columbus, OH 43210-1174, USA} \email{jlafont@math.ohio-state.edu }

\author{Ashot Minasyan}
\address{School of Mathematics, University of Southampton, Highfield campus,
Southampton, SO17 1BJ, United Kingdom} \email{aminasyan@gmail.com}

\begin{document}

\begin{abstract}
 We construct a class of finitely presented groups where the
  isomorphism problem is solvable but the commensurability
  problem is unsolvable. Conversely,
  we construct a class of finitely presented
  groups within which the commensurability problem is solvable but the isomorphism problem is unsolvable.
  These are first examples of such a contrastive
  complexity behavior with respect to the isomorphism problem.
\end{abstract}

\maketitle

\pagestyle{myheadings}
\markright{Isomorphism versus commensurability for a class of finitely presented groups}


\section{Introduction}

The purpose of this paper is to study the relative algorithm
complexities of the following two major group theoretical decision
problems: the isomorphism problem and the commensurability problem.

Both of these problems have a long history~\cite{Dehn,siegel}, a
meaningful topological interpretation~\cite{stillwell,borel}, and a
number of famous solutions for specific classes of
groups~\cite{DM,margulis, miller}. It is also well known that these problems are undecidable within the class of all finitely presented groups.
However their comparison from
the algorithmic point of view  seems not to have been done up to
now. Moreover, there have been numerous results comparing decision problems
dealing with elements in a single group, such as the word problem,
conjugacy problem, power problem, etc. (see, for
instance,~\cite{lipschutz-miller,miller}). In contrast, there have
so far been no comparative results involving the isomorphism
problem. We remedy this situation, by establishing the following two
complementary theorems:

\begin{theorem}\label{thm:main2}
There exists a recursively enumerable class $\mathcal C_1$ of finite
presentations of groups, with uniformly solvable word problem, such
that the isomorphism problem is solvable but the commensurability
problem is unsolvable within this class.
\end{theorem}

\begin{theorem}\label{thm:main1}
There exists a recursively enumerable class $\mathcal C_2$ of finite
presentations  of groups such that the commensurability problem is
solvable but the isomorphism problem is unsolvable within this
class.
\end{theorem}

These results are all the more unexpected as Thomas \cite[Thm. 1.1]{thomas1} showed that the isomorphism and commensurability problems
have the same complexity  from the viewpoint of descriptive set theory.

Let us now explain the terminology and the meaning of our main theorems.

A class $\mathcal C$ of finite presentations of groups has
\textit{uniformly solvable word problem} if there is an algorithm
which takes as input a presentation $P\in \mathcal{C}$ and a word in
the generators of this presentation, and decides whether or not this
word represents the identity element of the group given by $P$.

Two groups $G_1,\ G_2$ are \emph{commensurable}  if there exist two
subgroups of finite index $H_i\leqslant G_i$ for $i=1,2$, such that
$H_1$ and $H_2$ are isomorphic. It is not difficult to see that
commensurability is an equivalence relation.

Given a class $\mathcal C$ of finite presentations of groups, we say that the \textit{isomorphism problem is solvable within $\mathcal C$}
[\textit{commensurability problem is solvable within $\mathcal C$}] if there is an algorithm, taking on input two group presentations from $\mathcal C$ and deciding whether or not these
presentations define isomorphic [commensurable] groups.

Often, when considering the isomorphism problem, one is
looking at a certain class $\mathcal G$ of finitely presented {\it groups}. This actually means the class of
all finite presentations of groups from $\mathcal G$. At first glance, it might seem that our Theorems
\ref{thm:main2} and \ref{thm:main1} are somewhat more restrictive, as we are only picking out some
specific family of presentations. Let us clarify this issue.

 Let $\mathcal G_1$ denote the collection of all groups
defined via the presentations in the class $\mathcal C_1$ appearing in Theorem \ref{thm:main2}, and
let $\widehat {\mathcal C}_1$ denote the class of {\it all} finite presentations of groups from $\mathcal G_1$
(so clearly $\mathcal C_1 \subset \widehat {\mathcal C}_1$). It follows immediately from
Theorem \ref{thm:main2} that the commensurability problem is unsolvable for the class
$\widehat {\mathcal C}_1$ of finite presentations of groups, as it is already unsolvable within the
subclass $\mathcal C_1$. On the other hand, the isomorphism problem is still solvable within the
class $\widehat {\mathcal C}_1$. Indeed, given any presentation $P \in \widehat {\mathcal C}_1$, one can start applying Tietze
transformations to it; simultaneously we can start writing down the finite presentations from $\mathcal{C}_1$, because
the class $\mathcal{C}_1$ is recursively enumerable. At each step we can compare the transformations of $P$, obtained so far, with the presentations from
the class $\mathcal{C}_1$, written down by this step. After finitely many steps we will find a finite presentation $P' \in \mathcal{C}_1$
which defines the same group (up to isomorphism) as $P$ (see \cite[II.2.1]{L-S})).
This easily yields an algorithm that identifies a pair of presentations from $\mathcal C_1$
which define the same groups as the given pair of presentations in $\widehat {\mathcal C}_1$. Taking the
resulting pair of presentations in $\mathcal C_1$, we can then apply the
algorithm for deciding the isomorphism problem within the subclass $\mathcal C_1$.
As such, we  view Theorem~\ref{thm:main2} as a statement about the corresponding class of
groups $\mathcal G_1$.

Similarly, let $\mathcal G_2$ denote the collection of all groups
defined via the presentations in the class $\mathcal C_2$ appearing in Theorem \ref{thm:main1}, and
let $\widehat {\mathcal C}_2$ denote the class of {\it all} finite presentations of groups from $\mathcal G_2$
(so again, we have $\mathcal C_2 \subset \widehat {\mathcal C}_2$). By an argument, identical to the one
in the previous paragraph, we have that the isomorphism problem is unsolvable in the class
$\widehat {\mathcal C}_2$, but the commensurability problem is solvable. This allows us to view
Theorem \ref{thm:main1} as a statement about the corresponding class of
groups $\mathcal G_2$.

 The fact that the isomorphism problem is unsolvable within the class $\mathcal{C}_2$
implies, in particular, that there are infinitely many pairwise non-isomorphic
groups within $\mathcal G_2$. More precisely, the set of representatives of isomorphism classes of groups from $\mathcal{G}_2$ is
not recursively enumerable. This fact is of particular interest because it cannot be seen directly from our construction of the class $\mathcal{C}_2$ below.

The proofs of both theorems rely on a combination of various embedding theorems from
combinatorial and geometric group theory involving finitely presented infinite simple groups and infinite groups with no finite quotients.
The main idea is to start with a single group $G$ and construct a class $\mathcal{K}$, of
mapping tori of $G$,  for which the isomorphism problem is directly
related to the word problem in $G$. Similarly, the commensurability problem in $\mathcal{K}$
will be directly related to the torsion problem in $G$.
Thus the solvability/unsolvability of the word [resp. torsion] problem in $G$ will yield the same
for the isomorphism [resp. commensurability] problem in $\mathcal{K}$.

In Section \ref{sec:memb} we prove that there exist no recursive classes of groups with decidable isomorphism or commensurability problems. This shows that
the statements of Theorems \ref{thm:main2} and \ref{thm:main1} are optimal, as the recursively enumerable classes 
of groups we construct in these theorems cannot be recursive..

Besides isomorphism and commensurability, there are other natural equivalence relations on the class of finitely presented groups such
as virtual isomorphism, bi-Lipschitz equivalence, quasi-isometry, etc.
We discuss the corresponding algorithmic problems in the last section, where we also state some open questions.


\section{Mapping tori of groups without proper finite index subgroups}\label{sec:sdirect}

Let $G$ be a group and $\varphi\in \hbox{Aut}(G)$ be an automorphism
of $G$. Let $G_\varphi:=G\rtimes_{\varphi}\mathbb{Z}$ denote the
associated \textit{mapping torus}. As a set,
$G_\varphi=G\times\mathbb{Z}$ and the group product is defined by
$(g,n)(g',m):=(g\cdot \varphi^n(g'),{n+m}),$ where $\varphi^n$ denotes the
automorphism of $G$ which is the $n$-fold composition of $\varphi$, with the convention that $\varphi^0=id_G$, where
$id_G \in \mbox{Aut}(G)$ the identity automorphism of $G$.

We shall consider the class of groups $\mathcal{K}_{G,\Phi}=\left\{ G_{\varphi} \mid \varphi \in
\Phi\right\}$, where $\Phi$ is some subset of ${\rm Aut}(G)$, and analyze the isomorphism problem within the corresponding class of group presentations.
We denote by $\overline{\varphi}$ the image
of $\varphi$ under the canonical epimorphism
$\hbox{Aut}(G)\twoheadrightarrow
\hbox{Out}(G):=\hbox{Aut}(G)/\hbox{Inn}(G)$ onto the quotient of
$\hbox{Aut}(G)$ by the subgroup consisting of inner automorphisms.

\begin{proposition}\label{prop:simple1}
Suppose that $G$ is a group which has no epimorphisms onto $\Z$, and $\varphi, \psi \in \hbox{\rm Aut}(G)$.
Then the following are equivalent.
\begin{itemize}
  \item[(i)] $G_{\varphi}$ is isomorphic to $G_{\psi};$
  \item[(ii)] $\overline{\varphi}\in \hbox{\rm Out}(G)$ is conjugate to
  one of the two elements $\overline{\psi}, {\overline{\psi}}^{-1}\in \hbox{\rm Out}(G).$
\end{itemize}
\end{proposition}

\begin{proof}
Suppose that $G_{\varphi}$ is isomorphic to $G_{\psi}$ via an isomorphism
$$\rho\colon G\rtimes_{\varphi}\mathbb{Z} \longrightarrow G\rtimes_{\psi}\mathbb{Z}.$$

Let $\tau\colon G\to \mathbb{Z}$ be the homomorphism defined by the composition
$$
G\hookrightarrow G\rtimes_{\varphi}\mathbb{Z}\stackrel{\rho}{\longrightarrow} G\rtimes_{\psi}\mathbb{Z}\twoheadrightarrow \mathbb{Z},
$$
with the natural inclusion and epimorphism maps. It follows that $\tau$ is trivial as by hypothesis we know that $G$ does not map onto $\Z$.
Therefore, the restriction of $\rho$ to $G$ has image entirely contained in $G\leqslant G_{\psi}.$ Applying the same argument to $\rho^{-1}$ and
recalling that $\rho^{-1} \circ \rho = id_{G_\varphi}$, we can conclude that $\rho$ maps the $G$-factor in $G_{\varphi}$
isomorphically onto the $G$-factor in $G_{\psi}.$

On the other hand, a generator $t$ of the $\mathbb{Z}-$factor in $G_{\varphi}=G\rtimes_{\varphi}\mathbb{Z}$ has to map
to a generator under the composition
$$
\langle\,t\,\rangle=\mathbb{Z}\hookrightarrow G\rtimes_{\varphi}\mathbb{Z}\stackrel{\rho}{\longrightarrow}
G\rtimes_{\psi}\mathbb{Z}\twoheadrightarrow \mathbb{Z}.
$$

Indeed,  the composition map
$G_{\varphi}\stackrel{\rho}{\longrightarrow}G_{\psi}\twoheadrightarrow \mathbb{Z}$ is surjective as $\rho$ is an isomorphism.
Since $G\leqslant G_{\varphi}$ is contained in the kernel of this map, the image is
determined by the image of the quotient group $G_{\varphi}/G$. However,
such an image coincides with $\langle\,t\,\rangle$ through the short exact sequence
$\{1\}\to G \hookrightarrow G_{\varphi} {\twoheadrightarrow} \mathbb{Z} \to \{1\}.$
Thus, the surjectivity of $G_{\varphi}\stackrel{\rho}{\longrightarrow}G_{\psi}\twoheadrightarrow \mathbb{Z}$ implies that
$t\in G_{\varphi}$ maps to a generator $s^{\pm 1}$ of the $\mathbb{Z}-$factor
in $G_{\psi}.$

Thus, in terms of splittings, the isomorphism $\rho$
is of the form:
\begin{eqnarray*}
  (x,0) &\stackrel{\rho}{\mapsto}& (\alpha(x), 0) \\
  (e,1) &\stackrel{\rho}{\mapsto}& (g,{\pm 1})
\end{eqnarray*}
for any $x\in G$ and some fixed $\alpha\in \hbox{Aut}(G)$ and $g\in
G$ ($e\in G$ is the identity element).

Let us now focus on the case where $(e,1)\stackrel{\rho}\longmapsto
(g, 1)$. Since the map $\rho$ is assumed to be an isomorphism, it
must preserve the relations of the group $G_\varphi$. Evaluating
$\rho$ on the relation $(e,1)(x,0)(e,1)^{-1}=(\varphi(x),0)$ yields
the required constraint on the automorphisms. Indeed, evaluating the
left hand side, we obtain
$$
\rho\big((e,1)(x,0)(e,1)^{-1} \big)= (g,1)(\alpha(x),0)(\psi^{-1}(g^{-1}),-1) = \big(g \psi(\alpha(x)) g^{-1},0\big),
$$
while evaluating the right hand side, we obtain
$$
\rho\big((\varphi(x),0)\big) = \big(\alpha(\varphi(x)),0)\big).
$$
We deduce that the automorphism $\alpha\in \hbox{Aut}(G)$ and the
element $g\in G$ are related to the given automorphisms $\varphi,
\psi \in \hbox{Aut}(G)$ as follows:
$$
\alpha\circ\varphi = c_g\circ\psi\circ\alpha,
$$
where $c_g \in \hbox{Aut}(G)$ is the inner automorphism defined by $c_g(y)=gyg^{-1}$ for all $y \in G$.

Passing to the outer automorphism group, we see that we have to
have $\overline{\alpha}\circ\overline{\varphi}=\overline{\psi}\circ\overline{\alpha}$,
that is, the classes $\overline{\varphi}$ and $\overline{\psi}$ are conjugate in $\hbox{Out}(G).$

\smallskip
Conversely, if the classes $\overline{\varphi}, \overline{\psi} \in
\hbox{Out}(G)$ are conjugate by some $\overline{\alpha}\in
\hbox{Out}(G)$, then one can find an element $g\in G$ so that
$\alpha\circ\varphi = c_g\circ\psi\circ\alpha.$ It is now immediate
that $G_{\varphi}\cong G_{\psi},$ via the isomorphism map defined by
$(x,0) \mapsto (\alpha(x), 0)$ and  $(e,1) \mapsto (g,1).$

\smallskip

A similar analysis can be done in the case $(e,1)
\stackrel{\rho}{\longmapsto} (g,{-1})$. This yields the relation
$\alpha\circ\varphi = c_g\circ\psi^{-1}\circ\alpha,$ that is, the
classes $\overline{\varphi}$ and $\overline{\psi}^{-1}$ are
conjugate in $\hbox{Out}(G).$ This finishes the proof.
\end{proof}

In order to facilitate the notation let us give the following
\begin{definition} We will say that group $G$ is NFQ (\textit{`No Finite Quotients'}), if the only finite quotient of $G$ is the trivial group.
\end{definition}

Since every finite index subgroup contains a finite index normal subgroup, a group $G$ is NFQ if and only if $G$ has no proper subgroups of finite index.
It is easy to see that any NFQ group $G$ has no epimorphisms onto $\Z$, and thus it satisfies the assumptions of Proposition \ref{prop:simple1}.
Basic examples of NFQ groups are infinite simple groups.

To study the commensurability problem within the class
$\mathcal{K}_{G,\Phi}$, we need to know the structure of subgroups of
finite index in the corresponding mapping tori. The
following observation shows that all such subgroups are ``congruence
subgroups'':

\begin{proposition}\label{prop:simple2}
Let $G$ be a NFQ group, and $\varphi \in \hbox{\rm Aut}(G)$. Let
$\pi: G_\varphi \twoheadrightarrow \mathbb Z$ be the canonical
projection onto the $\mathbb Z-$factor of the mapping torus.
Assume that $H \leqslant G_\varphi$ is a finite index subgroup of
$G_\varphi$. Then $H= \pi ^{-1}\big(k \mathbb Z\big) \cong
G_{\varphi^k}$, where $k$ is the index of $H$ in $G_\varphi$ (and in
particular, $H$ must be normal in $G_\varphi$).
\end{proposition}

\begin{proof}
By the assumptions, $[G_\varphi:H]<\infty$, hence $[G:(G\cap H)]<\infty$, therefore $\ker \pi=G \leqslant H$ as $G$ is NFQ.
This forces $H$ to be of the form $\pi ^{-1}( k \mathbb Z)$ for
some $k$. The value of $k$ can then be easily deduced:
$$k=[\mathbb Z : k \mathbb Z] = [G_\varphi: \pi ^{-1}(k \mathbb Z)] = [G_\varphi : H],$$
as stated in the proposition.
\end{proof}

Combining Propositions~\ref{prop:simple1} and \ref{prop:simple2}, we immediately obtain

\begin{corollary}\label{cor:product-sbgrp}
Let $G$ be a NFQ group, and $\varphi \in \hbox{\rm Aut}(G)$. Then $G_\varphi$
is commensurable with $G_{id_G}\cong G \times \Z$ if and only if the element
$\overline{\varphi}\in \hbox{\rm Out}(G)$ has finite order.
\end{corollary}

Given two groups $A$ and $B$, consider their free product $G=A*B$. For any element $a \in A$ we can
define a natural automorphism  $\tau_a \in \mbox{\rm Aut}(G)$ by
$\tau_a(x):=a^{-1}xa$ for all $x \in A$ and $\tau_a(y):=y$ for all $y \in B$. Note that $(\tau_a)^k=\tau_{a^k}$
in ${\rm Aut}(G)$ for all $k \in \Z$.

\begin{lemma}\label{lem:tau-inner} Suppose that $B \neq \{1\}$, $a \in A$ and $G=A*B$. Then
$\tau_a \in {\rm Inn}(G)$ if and only if $a$ belongs to the center of $A$.
\end{lemma}

\begin{proof} Clearly, if $a$ is central in $A$, then $\tau_a=id_G \in {\rm Inn}(G)$. Conversely, suppose that there is
$c \in A$ such that $a^{-1}ca \neq c$ in $A$. Take any $b \in B \setminus \{1\}$
and consider the element $g:=cb \in G$. Then $\tau_a(g)=a^{-1}cab$ is not conjugate to $g$ in $G=A*B$
by the criterion of conjugacy in free products (see \cite[IV.1.4]{L-S}).
Hence $\tau_a \notin {\rm Inn}(G)$, as required.
\end{proof}

Since the free product of two NFQ groups is again a NFQ group, we can put together Proposition \ref{prop:simple1},
Corollary \ref{cor:product-sbgrp} and Lemma \ref{lem:tau-inner} to achieve

\begin{corollary}\label{cor:main_aux} Let $A$ and $B$ be NFQ groups such that $B \neq \{1\}$ and $A$ has trivial center.
Then for $G=A*B$ and any  $a \in A$ the following are true:
\begin{itemize}
    \item $G_{\tau_a}$ is isomorphic to $G_{id_G}$ if and only if $a=1$ in $A$;
    \item $G_{\tau_a}$ is commensurable with $G_{id_G}$ if and only if $a$ has finite order in $A$.
\end{itemize}

\end{corollary}


\section{Word and torsion problems in NFQ groups}\label{sec:conj-ppp}
For a finite set $X$, we use $X^*$ to denote the set of words with letters from $X^{\pm 1}$.
Let $R$ be a  set of words from $X^*$ and suppose that $G$ is a group
given by the presentation $P=\langle \, X \, \| \, R \, \rangle$.

For a subset $Z \subseteq X^*$, we say the \textit{word problem for $Z$ in $G$ is solvable}
if there is an algorithm, which takes on input a word
$w \in Z$ and decides whether or not this word represents the identity element of $G$.
If $Z=X^*$, then the word problem for $Z$ in $G$ is simply known as {\it the word problem in $G$}.
The word problem is one of the three fundamental group-theoretical decision problems
introduced by Max Dehn \cite{Dehn} in 1911 (other two being the conjugacy and the isomorphism problems). It is well known that if the word problem for $G$ is solvable with
respect to one finite generating set, then it is solvable with respect to any other finite generating set of $G$.

For an arbitrary subset $Z \subseteq X^*$,  one can also consider the {\it torsion problem for $Z$ in $G$}, asking whether there exists an
algorithm which inputs a word $w \in Z$, and decides whether or not $w$ represents an element of finite order in $G$.
This is closely related to some decision problems considered
by Lipschutz and Miller in \cite{lipschutz-miller} (for instance, it is a special case of the
{\it power problem}).

\begin{proposition}\label{prop:NFQ-emb} Every finitely presented group $H$ can be embedded into a finitely
presented NFQ group $A$ with trivial center. Moreover, if the word problem in $H$ is solvable then it is also
solvable in $A$.
\end{proposition}

\begin{proof} Take any infinite finitely presented simple group $S$ (for instance, Thompson's group $T$ or $V$~\cite{cannon-floyd-parry},
or see \cite{burger-mozes-1, burger-mozes-2,caprace-remy-1,caprace-remy-2} for other such
groups) and consider the free product $G=S*S$.
Then $G$ is NFQ and hyperbolic relative to these two copies of $S$. Therefore,
by Theorem 1.1 from \cite{SQ}, $H$ can be isomorphically embedded into some quotient $Q$ of $G$.
Moreover, from the proof of this theorem, it follows that
$Q$ can be obtained from $H*G$ by adding only finitely many defining relations. Consequently, as both $H$ and $G$
are finitely presented, $Q$ will also be finitely presented.
The group $Q$ is NFQ as a quotient of the NFQ group $G$. One can check that the center of the group $Q$,
obtained from \cite[Thm. 1.1]{SQ}, is in fact trivial.
However, it is easy to bypass this, by setting $A:=Q*S$ and observing that $A$ is still finitely presented, NFQ,
has trivial center (as a non-trivial free product -- see \cite[6.2.6]{Robinson})
and contains a copy of $H$.

Now, suppose that the word problem in $H$ is solvable. Note that the same is true in $S$, because the
word problem is solvable in any recursively presented simple group (\cite[IV.3.6]{L-S}).
By \cite[Thm. 1.1]{SQ} the group $Q$ above is hyperbolic relative to the family of subgroups, consisting of $H$
and two copies of $S$. Therefore $Q$ has solvable word problem (see \cite[Thm. 3.7]{Farb} or \cite[Cor. 5.5]{Osin-rel_hyp_geom}).
Finally, the word problem is solvable in $A=Q*S$ by \cite[IV.1.3]{L-S}.
\end{proof}

\begin{proposition}\label{prop:first_constr} There exists a NFQ group $A_1$, with trivial center and finite presentation $P_1=\langle  X_1 \, \|\, R_1  \rangle$,
and a recursively enumerable subset of words $Z_1=\{z_1,z_2,\dots\} \subset X_1^*$ such that the word problem in $A_1$ is solvable but the torsion problem for $Z_1$ in $A_1$ is unsolvable.
\end{proposition}

\begin{proof} Let $H_0$ be the center-by-metabelian group constructed by P. Hall in \cite[p. 435]{P.Hall}. Namely, $H_0$ is generated by two elements $a,b$, subject to the relations
$$[[b_i,b_j],b_k]=1, \mbox{ for } i,j,k=0,\pm 1,\pm 2,\dots, \mbox{ where } b_i:=a^{-i}ba^{i}, [x,y]:=x^{-1}y^{-1}xy, \mbox{ and }$$
$$c_{i,j}=c_{i+k,j+k}, \mbox{ for } j>i, \quad i,j,k=0,\pm 1,\pm 2,\dots, \mbox{ where } c_{i,j}:=[b_j,b_i].$$
As Hall proved in \cite[p. 435]{P.Hall}, the center of $H_0$ is the free abelian group with free abelian basis $\{d_1,d_2,\dots\}$, where $d_r:=c_{0,r}=[a^{-r}ba^{r},b]$, $r=1,2,\dots$.

Let $\langle a,b \,\|\, R_0\rangle$ be the above presentation for $H_0$. Clearly this presentation is recursive.
Now, consider a computable (recursive) function $f: \mathbb{N} \to \mathbb{N}$ with non-recursive range $f(\N) \subset \mathbb{N}$.
Let $H_1$ be the quotient of $H_0$ by the central subgroup $\langle d_{f(n)}^n \mid n \in \N\rangle$ where $d_r$, $r \in \N$, are as above.
Then $H_1$ has the presentation $$\left\langle a,b \, \left\| \, R_0,\left([a^{-f(n)}ba^{f(n)},b]\right)^n, \right. n \in \N \right\rangle .$$
The group $H_1$ will be recursively presented since $R_0$ is recursively enumerable and $f$ is computable.

By abusing notation, we will continue writing $a,b,b_i,d_r$ for the images of the corresponding elements of $H_0$ in $H_1$.
We can solve the word problem in $H_1$ as follows. Given a word $w$, over the alphabet $\{a^{\pm 1},b^{\pm 1}\}$,
we want to determine whether $w=1$ in $H_1$. First we compute the sum $\varepsilon_a(w)$ of all exponents of $a$ in $w$. If $\varepsilon_a(w) \neq 0$, then $w \neq 1$ in $H_1$ as
there is a homomorphism $\alpha:H_1 \to \langle a \rangle$, whose kernel is generated by $b_i$, $i \in \Z$, such that $\alpha(w)=a^{\varepsilon_a(w)} \neq 1$. If $\varepsilon_a(w) = 0$, then
$w \in B:=\langle b_i,i \in \Z\rangle$ and we can re-write $w$ as a word $w_1$ in letters $b_i$, $i \in \Z$. If for some $i \in \Z$, $\varepsilon_{b_i}(w_1) \neq    0$, then, again,
$w \neq 1$ in $H_1$, because its image will be non-trivial in the abelianization of $B$. Otherwise, $w$ will represent an element of the center $C:=\langle d_r, r \in \N\rangle$ of $H_1$,
and we can re-write $w_1$ as a word $w_2\equiv d_{r_1}^{n_1}d_{r_2}^{n_2} \cdots d_{r_l}^{n_l}$, where $l\geqslant 0$, $1 \leqslant r_1<r_2<\dots<r_l$, and $n_j \in \Z\setminus\{0\}$ for $j=1,\dots,l$.
Note that $C=\bigoplus_{r \in \N} \langle d_r \rangle$ by definition. If $l=0$ then $w=w_2=1$ in $H_1$. If $l>0$, then $w_2=1$ in $C$ if and only if the order of $d_{r_j}$ in $H_1$ divides $n_j$
for all $j=1,2,\dots,l$. The latter can be verified as follows: for every positive divisor $m$ of $n_j$, we compute $f(m)$ and check if it is equal to $r_j$. If this happens for some
divisor $m$ of $n_j$, then the order of $d_{r_j}$ in $H_1$ is $m$, by construction, and so $d_{r_j}^{n_j}=1$. If this is true for all $j=1,\dots,l$, then $w=w_2=1$ in $H_1$.
As each $n_j$ has only finitely many divisors, this can be checked in
finitely many steps. Finally, if there is $j \in \{1,\dots,l\}$ such that for every positive divisor $m$ of $n_j$, $f(m) \neq r_j$, then the order of $d_{r_j}$ in $H_1$ does not divide
$n_j$, and hence $w=w_2 \neq 1$ in $H_1$.

Thus $H_1$ is a finitely generated recursively presented group with solvable word problem. By a theorem of Clapham \cite[Thm. 6]{Clapham}, $H_1$ can be embedded in a finitely presented group
$H_2$ with solvable word problem. Now we can use Proposition \ref{prop:NFQ-emb} to embed $H_2$ into a finitely presented NFQ group $A_1$, with trivial center and solvable word problem.
Let $P_1=\langle \, X_1 \, \| \, R_1  \, \rangle$ be some finite presentation for $A_1$. Since $H_1 \leqslant A_1$, the generators $a,b$ of $H_1$ can be represented by
some words $w_1$, $w_2$ (respectively) in the alphabet $X_1^{\pm 1}$, and hence every word in letters from
$\{a^{\pm 1},b^{\pm 1}\}$ can be effectively re-written in letters from  $X_1^{\pm 1}$. So, for every $r \in \N$ we can effectively compute
a word $z_r\in X_1^*$ representing $d_r$ in $A_1$ and set $Z_1:=\{z_r \mid r \in \N\} \subset X_1^*$. By construction, $Z_1$ is recursively enumerable.

Suppose that the torsion problem for $Z_1$ in $A_1$ is solvable. Then for any $r \in \N$ we can compute the word $z_r \in Z_1$, representing $d_r$ in $A_1$, and check if $d_r$
has finite order in $A_1$. But the latter happens if and only if $r \in f(\N)$. Thus we would be able to determine whether or not $r$ belongs to the range of $f$, contradicting to the choice of $f$.
Therefore the torsion problem for $Z_1$ in $A_1$ is unsolvable and the proposition is proved.
\end{proof}

The next statement suggests a construction which is in some sense opposite to the construction of Proposition \ref{prop:first_constr}.

\begin{proposition}\label{prop:second_constr} There exists a NFQ group $A_2$, with trivial center and finite presentation $P_2=\langle X_2 \, \|\, R_2 \rangle$,
and a recursively enumerable subset of words $Z_2 =\{z_1,z_2,\dots\} \subset X_2^*$ such that every word from $Z_2$ represents an element of order at most $2$ in $A_2$ but
the word problem for $Z_2$ in $A_2$ is unsolvable.
\end{proposition}

\begin{proof} Again, let us start with Hall's group $H_0$, used in the proof of Proposition  \ref{prop:first_constr}, keeping the same notation as before. Let $f:\N \to \N$ be a computable
function with non-recursive range. We now let $H_1$ be the quotient of $H_0$ by the central subgroup $\langle d_n^2,d_{f(n)} \mid n\in \N\rangle$.

As before, $H_1$ will be finitely generated and recursively presented, however, the word problem in $H_1$ will be unsolvable (since the set $f(\N)$ is not recursive).
By the celebrated theorem of Higman \cite{Higman}, one can embed $H_1$ into a finitely presented group $H_2$, and applying Proposition \ref{prop:NFQ-emb}, we can embed $H_2$ into a finitely presented
NFQ group $A_2$ with trivial center.

Let $P_2=\langle \, X_2 \, \| \, R_2  \, \rangle$ be some finite presentation for $A_2$.
Fix some words $w_1,w_2\in X_2^*$ representing the generators $a$, $b$ (respectively) of $H_1$ in $A_2$. Clearly there is an algorithm which takes on input a word in the alphabet $\{a^{\pm 1},b^{\pm 1}\}$ and outputs
a corresponding word in the alphabet $X_2^{\pm 1}$ (substituting every $a$-letter by $w_1$ and every $b$-letter by $w_2$). For each $r\in \N$, let $z_r \in X_2^*$ be the
word representing $d_r \in H_1$, obtained this way, and set $Z_2:=\{z_r \mid r \in \N\}\subset X_2^*$. Evidently the set of words $Z_2$ is recursively enumerable and every word from this set represents an element
$d_r$, which has order at most $2$ in $A_2$.
By construction, $z_r=1$ in $A_2$ if and only if $d_r=1$ in $H_1$, which happens if and only if
$r \in f(\N)$. Since $f(\N)$ is non-recursive, we see that the word problem for $Z_2$ in $A_2$ is unsolvable.
\end{proof}


\section{Proofs of the theorems}\label{sec:result}

We are now ready to establish our two theorems.

\begin{proof}[Proof of Theorem~\ref{thm:main2}] We start with the presentation $P_1=\langle  X_1 \, \|\, R_1 \rangle$ of the group $A_1$, and the recursively enumerable set of words $Z_1=\{z_1,z_2,\dots,\} \subset X_1^*$, which were
constructed in Proposition~\ref{prop:first_constr}.
Take some infinite finitely presented simple group $B$ and fix some finite presentation $\langle Y \, \| \, S\rangle$ of it;
recall that the word problem in $B$ is solvable by \cite[IV.3.6]{L-S}.
Let $z_0 \in X_1^*$ be the empty word. For each $r \in \N\cup \{0\}$, let  $d_r$ denote the element of $A_1$ represented by the $z_r \in Z_1$; let $G:=A_1*B$ and
let $C_{r+1}$ be the cyclic group of order $r+1$. Then the group
$K_r:=G_{\tau_{d_r}} \times C_{r+1}$ has the
presentation
\begin{multline*} P_{1,r}:=\left\langle X_1,Y,t,u \, \left \|\, R_1,S,\; t^{-1} x^{-1} t z_r^{-1} x z_r, \; t^{-1}y^{-1} ty,\; u^{-1}x^{-1} ux, \right.\right.\\
\left. u^{-1}y^{-1}uy,\;  u^{-1}t^{-1}ut, \; u^{r+1}, \mbox{ for all } x \in X_1
\mbox{ and } y \in Y \right\rangle.\end{multline*}

Since the sets $X_1$, $Y$, $R_1$ and $S$ are finite, for every $r \in \N\cup \{0\}$, $P_{1,r}$ is a finite presentation of a group.
Note that the presentation $P_{1,0}$ defines the group $K_0 \cong G_{id_G} \cong G \times \Z$.

Now, consider the class of group presentations $\mathcal{C}_1:=\{P_{1,r} \mid r \in\N\cup \{0\}\}$. We can make the following observations.
\begin{description}
  \item[(a)] the class $\mathcal{C}_1$ is recursively enumerable by definition.
  \item[(b)] the word problem in $\mathcal{C}_1$ is uniformly solvable. This easily follows from the fact that the word problem
  in $G=A_1*B$ is solvable and for each $r \in \N\cup\{0\}$, $G \lhd K_r$ and
        $K_r/G \cong \Z \times C_{r+1}$.
  \item[(c)] the isomorphism problem within  $\mathcal{C}_1$ is trivially solvable. This is because for any $r \in \N\cup \{0\}$, the abelianization of the group $K_r$ is isomorphic to $\Z\times C_{r+1}$
        (as $G$ is NFQ), hence for any $q \in \N\cup \{0\}$, $q\neq r$, the group $K_r$ is not isomorphic to $K_q$ since their abelianizations have different torsion subgroups.
        Thus any two distinct presentations from $\mathcal{C}_1$ define non-isomorphic groups.
  \item[(d)] the commensurability problem within $\mathcal{C}_1$ is unsolvable. Indeed, since the index $[K_r:G_{\tau_{d_r}}]=r+1$ is finite, the group $K_r$ is commensurable with the group $G_{\tau_{d_r}}$ for each
        $r \in \N\cup\{0\}$. So, if we could decide
        whether $K_r$ is commensurable with $K_0$, then we would be able to decide whether $G_{\tau_{d_r}}$ is commensurable with $G_{id_G}$, which, by Corollary \ref{cor:main_aux},
        would imply that the torsion problem for $Z_1$ in $G=A_1*B$ is solvable, contradicting to the claim of Proposition \ref{prop:first_constr}.
\end{description}

Thus the class of group presentations $\mathcal{C}_1$ satisfies all of the required properties.
\end{proof}

\begin{proof}[Proof of Theorem~\ref{thm:main1}]
Now  we start with the presentation $P_2=\langle  X_2 \, \|\, R_2 \rangle$ of the group $A_2$,
constructed in Proposition~\ref{prop:second_constr}, and the recursively enumerable set of words $Z_2=\{z_1,z_2,\dots,\} \subset X_2^*$.
Take some infinite finitely presented simple group $B$ and fix some finite presentation $\langle Y \, \| \, S\rangle$ of it; then $B$ will have solvable word problem (\cite[IV.3.6]{L-S}).
Let $z_0 \in X_2^*$ be the empty word. For each $r \in \N\cup \{0\}$, let  $d_r$ denote the element of $A_2$ represented by the $z_r \in Z_2$ and let $G:=A_2*B$.
Then the group $G_{\tau_{d_r}}$ has the
presentation
$$ P_{2,r}:=\left\langle X_2,Y,t \, \left \|\, R_2,S,\; t^{-1} x^{-1} t z_r^{-1} x z_r, \; t^{-1}y^{-1} ty,  \mbox{ for all } x \in X_2
\mbox{ and } y \in Y\right. \right\rangle.$$

Since the sets $X_2$, $Y$, $R_2$ and $S$ are finite, for every $r \in \N\cup \{0\}$, $P_{2,r}$ is a finite presentation of a group.
As before, the presentation $P_{2,0}$ defines the group $G_{id_G} \cong G \times \Z$.

For the class of finite presentations $\mathcal{C}_2:=\{P_{2,r} \mid r \in\N\cup \{0\}\}$ we can observe the following.
\begin{description}
  \item[(a)] the class $\mathcal{C}_2$ is recursively enumerable by definition.
  \item[(b)] the commensurability problem within $\mathcal{C}_2$ is trivially solvable, because any presentation from this class defines the group $G_{\tau_{d_r}}$, for some $r \in \N\cup\{0\}$,
        which is commensurable with $G_{id_G}$ by Corollary \ref{cor:main_aux}, as the element $d_r\in A_2$ has finite order by construction. Thus any two presentations from
         $\mathcal{C}_2$ define commensurable groups.
  \item[(c)] the isomorphism problem within  $\mathcal{C}_2$ is unsolvable. Indeed, according to Corollary   \ref{cor:main_aux}, for any $r \in \N$, the group $G_{\tau_{d_r}}$,
        defined by $P_{2,r}$ is isomorphic to $G_{id_G}$, defined by $P_{2,0}$, if and only if $z_r=1$ in $G$. Thus the isomorphism problem within $\mathcal{C}_2$ is equivalent to the
     word problem for $Z_2$ in $A_2$, which is unsolvable by construction.
\end{description}

Thus the class $\mathcal{C}_2$ satisfies all of the needed properties.
\end{proof}

\section{Membership problem for some classes of groups}\label{sec:memb}
The purpose of this section is to show that the claims of Theorems \ref{thm:main2} and \ref{thm:main1} are optimal; that is, the 
recursively enumerable classes of groups we construct in these theorems cannot be recursive. More precisely, suppose that we are given a class of finitely presented
groups $\mathcal{K}$, closed under isomorphism.  The class $\mathcal K$ is said to be \emph{recursive}, or, equivalently, 
the \emph{membership problem to $\mathcal{K}$ is decidable} (within the class of all finitely presented groups), if there is an algorithm, which takes on input a finite 
presentation and decides whether or not the group defined by this presentation belongs to $\mathcal K$.

The next statement essentially shows that there is no recursive class of groups satisfying the claim of Theorem \ref{thm:main2}.

\begin{proposition}\label{prop:no_rec-1} Let $\mathcal K$ be a non-empty class of finitely presented groups with solvable isomorphism problem. Then $\K$ is not recursive.
\end{proposition} 

\begin{proof} Arguing by contradiction, suppose that $\K$ is a non-empty recursive class of finitely presented groups such that the isomorphism problem is solvable within $\K$.
Since $\K$ is non-empty, we can assume that we possess a finite presentation of some group $G \in K$. Let $A$ be an arbitrary finitely presented group and set $H:=G*A$. 
Then $H$ also finitely presented
and a finite presentation  of $H$ can be easily obtained from the finite presentations of $G$ and $A$. Now, by the assumptions, we can decide whether or not $H$ 
belongs to $\K$. If $H \notin \K$, then $H$ is not isomorphic to $G$, hence the group $A$ is non-trivial. If $H \in \K$ then one can apply the algorithm solving the isomorphism
problem in $\K$ to decide whether or not $H \cong G$. Evidently, if $H$ is not isomorphic to $G$ then $A$ is non-trivial. On the other hand, if $H \cong G$ then 
$\rank(H)=\rank(G*A)=\rank(G)+\rank(A)=\rank(G)$ by Grushko-Neumann theorem (see \cite[IV.1.9]{L-S}), where $\rank(G)$ denotes the minimal number of elements required to 
generate the group $G$. Hence $\rank(A)=0$, i.e., $A$ is the trivial group. Thus we have described an algorithm which decides whether any given finitely presented group $A$
is trivial. But it is well-known that the triviality problem is unsolvable within the class of all finitely presented groups (see, for example, \cite[Thm. 2.2]{rabin}).
This contradiction proves the claim of the proposition.
\end{proof}

Now we state a similar fact for the class of groups appearing in Theorem \ref{thm:main1}.
\begin{proposition}\label{prop:no_rec-2} Let $\mathcal K$ be a non-empty class of finitely presented groups with solvable commensurability problem. Then $\K$ is not recursive.
\end{proposition} 

The proof of this statement will utilize the following lemma:
\begin{lemma}\label{lem:comm-aux} Let $G$ and $A$ be groups and let $H:=G*(A*A)$, $F:=G \times (A*A)$. Suppose that $H$ and $F$ are both commensurable to $G$. Then
$A$ is the trivial group.
\end{lemma}

\begin{proof} First suppose that $G$ is finite. Since $F$ is commensurable to $G$, it must also be finite, hence $A*A$ is finite, 
which can only happen if $A$ is trivial. 

Assume, now, that $G$ is infinite. If $A$ is non-trivial then $A*A$ is infinite,
hence any finite index subgroup of $F$ itself contains a finite index subgroup that
decomposes as a non-trivial direct product (of a finite index subgroup in $G$ with a finite index subgroup in $A*A$). While any
finite index subgroup of $H$ decomposes in a non-trivial free product (by Kurosh theorem \cite[IV.1.10]{L-S}), and hence it cannot be isomorphic to a
non-trivial direct product (see, for example, \cite[Observation, p. 177]{L-S}). Therefore $F$ cannot be commensurable to $H$; this contradicts with the assumption that they
are both commensurable to $G$ and the fact that commensurability is a transitive relation. Thus $A$ must be trivial.
\end{proof}

\begin{proof}[Proof of Proposition \ref{prop:no_rec-2}] 
As before, assume that there is a recursive class $\K$ with solvable commensurability problem, and let $G \in \K$ be some group with a given finite presentation. 
Suppose we are given a finite presentation of any group $A$. Then we can easily produce the finite presentations for the groups $H:=G*A*A$ and $F:=G \times (A*A)$.

By the assumptions, we can decide whether or not $H \in \K$ and $F \in \K$. If at least one of these groups does not belong to $\K$, then $A$ is non-trivial. So, assume that 
both $H$ and $F$ lie in $\K$. Since the commensurability problem is solvable within $\K$, we can decide if $H$ and $F$ are commensurable to $G$. Evidently, if at least one 
of these groups is not commensurable to $G$ then $A$ is non-trivial. So, we can further suppose that both $H$ and $F$ are commensurable to $G$.
And Lemma \ref{lem:comm-aux} shows that the latter can happen only if $A$ is trivial. Thus, again, we produced an algorithm deciding the triviality of $A$, which leads us to the 
required contradiction.
\end{proof}

\section{Decision problems in geometric group theory}

From the viewpoint of geometric group theory, besides the isomorphism and commensurability relations,
there are several
other equivalence relations on groups which are of natural interest:\medskip

$\bullet$ Two finitely generated groups $G_1,G_2$ are {\it virtually
isomorphic} (sometimes also called {\it commensurable up to finite
kernels}) if there exist a pair of finite index subgroups $H_i
\leqslant G_i$, and some further finite normal subgroups
$N_i\trianglelefteq H_i$, $i=1,2$, with isomorphic quotients $H_1/N_1 \cong
H_2/N_2$.\medskip

$\bullet$ Two finitely generated groups $G_1,G_2$ are {\it
quasi-isometric} if there exists a map ${f: G_1\rightarrow G_2}$ and a
constant $K>0$ so that for all $x,y\in G_1$
$$\frac{1}{K} d_1(x,y) - K  \leqslant d_2\big(f(x), f(y)\big) \leqslant K \cdot d_1(x,y) + K  $$
and the $K$-neighborhood of $f(G_1)$ is all of $G_2$ (the $d_i$ are word metrics on the $G_i$, $i=1,2$).\medskip

$\bullet$ Two finitely generated groups are {\it bi-Lipschitz equivalent}
if there is a bi-Lipschitz map between $(G_1, d_1)$ and $(G_2,d_2)$, where
again the $d_i$ are word metrics (this is equivalent to the existence of a bijective
quasi-isometry between them - see Whyte \cite{whyte}).\medskip

We can now state the corresponding decision problems: the {\it
virtual isomorphism problem} (respectively, {\it quasi-isometry
problem} or {\it bi-Lipschitz problem}) asks whether there exists an
algorithm which, given two finite presentations of groups, can decide
whether or not they define virtually isomorphic (resp. quasi-isometric
or bi-Lipschitz equivalent) groups. Several of these problems
have been studied from the viewpoint of descriptive set theory by
Thomas~\cite{thomas2, thomas3, thomas1, thomas-velickovic}. Note
that a group is bi-Lipschitz equivalent to the trivial group if and
only it is trivial, and that it is virtually isomorphic,
commensurable, or quasi-isometric to the trivial group if and only
if it is finite. Since the problem of deciding whether a finitely
presented group is finite (or trivial) is unsolvable (this follows
from the famous Adian-Rabin theorem, see~\cite{adian1,adian2} and
\cite{rabin}), we immediately obtain

\begin{lemma}
Within the class of all finite presentations of groups, the virtual isomorphism, quasi-isometry, bi-Lipschitz,
and commensurability problems are all unsolvable.
\end{lemma}

It would be of some interest to study the relative complexity of these various decision problems.
A straightforward consequence of our construction appearing in
the proof of Theorem~\ref{thm:main1} is the following:

\begin{corollary}
There exists a recursively enumerable class of finite presentations of groups within which the isomorphism
problem is unsolvable, but the virtual isomorphism, quasi-isometry, and bi-Lipschitz
problems are all (trivially) solvable.
\end{corollary}

\begin{proof} In the notations from the proof of theorem \ref{thm:main1}, let $\Phi:=\left\{\tau_{d_r} \mid r \in \N\cup \{0\}\right\} \subset {\rm Aut(G)}$.
Then the class of finite presentations $\mathcal{C}_2$, constructed in the proof of Theorem~\ref{thm:main1}, defines the class
$\mathcal{K}_{G,\Phi}=\left\{G_{\tau_{d_r}} \mid r \in\N\cup\{0\}\right\}$ of finitely presented groups.
As we noticed above, any two groups from this class are commensurable. And since commensurable groups are automatically quasi-isometric, all
the groups in $\mathcal{K}_{G,\Phi}$ are quasi-isometric to each other, and the quasi-isometry problem within $\mathcal{C}_2$ is
(trivially) solvable.

Moreover, none of the groups $G_{\tau_{d_r}}$, $r \in \N\cup\{0\}$, can contain a non-trivial finite normal subgroup,
for such a subgroup would have to map to the identity under the
canonical projection $G_{\tau_{d_r}} \twoheadrightarrow \mathbb Z$, and
hence it would have to be a normal subgroup in the group $G=A_2*B$.
But a non-trivial free product does not have any non-trivial finite normal subgroups.
This tells us that within the class  $\mathcal{K}_{G,\Phi}$, two groups are
virtually isomorphic if and only if they are commensurable. Therefore
the virtual isomorphism problem within $\mathcal{C}_2$ is also (trivially) solvable.

Finally, noting that  $G=A_2*B$ is non-amenable, as a non-elementary free product, and embeds into every
$G_{\tau_{d_r}}$, $r \in\N\cup\{0\}$, we see that all the groups in the class
$\mathcal{K}_{G,\Phi}$ are non-amenable. The work of Block and Weinberger
\cite[Thm. 3.1]{block-weinberger} implies that the groups in this class
all have vanishing $0$-dimensional uniformly finite homology.
Whyte's thesis \cite[Thm. 1.1]{whyte} then implies that commensurability
between any two groups from $\mathcal{K}_{G,\Phi}$ can be promoted to a
bi-Lipschitz equivalence. We conclude that all the groups in $\mathcal{K}_{G,\Phi}$
are bi-Lipschitz equivalent to each other, so that
the bi-Lipschitz problem within $\mathcal{C}_2$ is also (trivially) solvable.
\end{proof}

More generally, we expect that these various decision problems are fundamentally unrelated
to each other (with the possible exception of the bi-Lipschitz problem, in view of Whyte's thesis \cite{whyte}).
To be more precise, we suspect that given any two disjoint subsets of
these decision problems, one can find a recursively enumerable class of finite presentations of groups such that
any problem from the first of these subsets is solvable within this class, while problems from the second subset are all unsolvable.

In another vein, these algorithmic problems are also open for various
natural classes of groups. For instance, one could focus on certain classes of lattices within a fixed
semi-simple Lie group $G$ of non-compact type. If the $\mathbb R$-rank of $G$ is $\geqslant 2$, and one
restricts to uniform lattices (so that the quasi-isometry
problem is trivially solvable), is the isomorphism problem or commensurability problem solvable?
If one focuses on $G=SO(n,1)$, $n\geqslant 4$, and restrict to non-uniform lattices (so that the isomorphism problem is
solvable, by Dahmani and Groves \cite{dahmani-groves}), is the commensurability problem (equivalent to the quasi-isometry
problem, by Schwartz \cite{schwartz}) solvable or not? Surprisingly, these questions do not seem to have been
considered in the literature.

\vskip 10pt

\centerline{\bf Acknowledgments}

\vskip 5pt

The authors thank Cornelia Dru\c tu, Daniel Groves, and
Mark Sapir for helpful conversations. The research of the first
author was partially supported by the ERC grant ANALYTIC no.
259527, and by the Swiss NSF, under Sinergia grant
CRSI22-130435. The second author was partially supported by the
NSF, under grant DMS-0906483, and by an Alfred P. Sloan research
fellowship. The work of the third author was supported by the EPSRC grant EP/H032428/1.

\bibliography{isom-7}
\bibliographystyle{plain}

\end{document}